\documentclass[a4paper,11pt]{amsart}
\usepackage{amssymb, graphics}

\newtheorem{lemma}{Lemma}

\newtheorem{prop}{Proposition}
\newtheorem{thm}{Theorem}

\newtheorem{conj}{Conjecture}
\theoremstyle{definition}
\newtheorem{rem}{Remark}
\newtheorem{defn}{Definition}

\newtheorem{prob}{Problem}

\newcommand{\R}{{\mathbb R}}
\newcommand{\C}{{\mathbb C}}

\newcommand{\cC}{{\mathcal C}}
\newcommand{\cO}{{\mathcal O}}

\newcommand{\Aut}{\mathrm{Aut}}

\begin{document}

\title{A splitting theorem for extremal K\"ahler metrics}

\author[Vestislav Apostolov]{Vestislav Apostolov}
\address{Vestislav Apostolov, D\'epartement de Math\'ematiques, UQAM, C.P. 8888, Succ. Centreville, Montr\'eal (Qu\'ebec), H3C 3P8, Canada}
\email{apostolov.vestislav@uqam.ca}

\author[Hongnian Huang]{Hongnian Huang}\address{Hongnian Huang, CMLS, \'Ecole Polytechnique, Palaiseau, 91128, France}
\email{hnhuang@gmail.com}

\date{}
\thanks{The authors would like to thank D. Phong for valuable suggestions, as well as X.~X.~Chen for his interest in this work. A special acknowledgment is due to G.~Sz\'ekelyhidi for suggesting to us  to consider approximations with balanced metrics, and for sharing with us his expertise. He decisively contributed to this project by pointing out to us the notion of Chow stability relative to a maximal torus discussed in the paper, as well as the uniqueness result in Lemma~2.}
\thanks{The second named author is financially supported by the Fondation math\'ematique Jacques Hadamard.}

\begin{abstract}
Based on recent work of S.~K.~Donaldson~\cite{Do-one,D1} and T.~Mabuchi~\cite{M1,M2,M3}, we prove that any extremal K\"ahler metric in the sense of E.~Calabi~\cite{cal}, defined on the product of polarized compact complex projective manifolds is the product of extremal K\"ahler metrics on each factor, provided that either the integral Futaki invariants of the polarized manifold vanish or its automorphism group satisfies a constraint. This extends a result of S.-T.~Yau~\cite{yau} about the splitting of a K\"ahler--Einstein metric on the product of compact complex manifolds to the more general setting of extremal K\"ahler metrics.
\end{abstract}


\maketitle

\section{Introduction} Extremal K\"ahler metrics were first introduced and studied by E.~Calabi in
\cite{cal,cal-2}. Let $X$ denote a connected compact complex manifold of complex dimension $n$.  A K\"ahler metric $g$ on $X$, with
K\"ahler form $\omega_g$, is {\it extremal} if it is a
critical point of the functional $g \mapsto \int _X s _g ^2 \, \frac{\omega
_g ^n}{n!}$, where $g$ runs over the set of all K\"ahler metrics on $X$
within a fixed K\"ahler class $\Omega = [\omega]$, and $s_g$ denotes the
scalar curvature of $g$. As shown in \cite{cal}, $g$ is extremal if and
only if the symplectic gradient $K := {\rm grad} _{\omega} s _g = J \, {\rm
grad} _g s _g$ of $s _g$ is a Killing vector field (i.e. $\mathcal{L} _K g
= 0$) or, equivalently, a (real) holomorphic vector field. Extremal K\"ahler metrics include K\"ahler
metrics of constant scalar curvature --- CSC K\"ahler metrics for short ---
and in particular K\"ahler--Einstein metrics.  Clearly, if the identity component ${\rm Aut} _0 (X)$ of the automorphism group of $X$ is
reduced to $\{1\}$, i.e. if $X$ has no non-trivial holomorphic vector fields, any extremal K\"ahler metric is CSC, whereas a CSC K\"ahler metric is K\"ahler--Einstein if and only if $\Omega$ is a multiple of the (real) first Chern class $c _1 (X)$.  

It is conceivable to think about an extremal K\"ahler metric $g$ in  $\Omega$  as a {\it canonical} representative of the K\"ahler metrics in the K\"ahler class $\Omega$. One would then expect that the extremal K\"ahler metrics in $\Omega$ reflect most of the holomorphic invariants of the pair $(X, \Omega)$.  In this vein, the goal of this note is  to discuss how  that the following natural splitting problem fits in  with some recent progress in the field.

\begin{prob}Let $X_i,$ be compact projective manifolds polarized by ample holomorphic line bundles $L_i$ and $X= \prod_{i=1}^r X_i$ their product endowed with the polarisation $L= \bigotimes_{i=1}^r L_i$, where $L_i$ is seen as a holomorphic line bundle over $X$ via the natural pull-back. Does any extremal K\"ahler metric $g$ in the K\"ahler class $\Omega = 2\pi c_1(L)$ on $X$ is the Riemannian product of extremal K\"ahler metrics $g_i$ in the K\"ahler classes $\Omega_i=2\pi c_1(L_i)$ on the factors $X_i$?\end{prob}

Several remarks are in order.

\smallskip
First of all, it is well-known (see e.g. \cite[Thm.~2.1]{yau}) that the answer is positive if we suppose that  $g$ is a K\"ahler--Einstein metric on $X$.  It then follows from a standard Bochner argument (see e.g. \cite{gauduchon-0,kob}) for the holomorphic projectors $P_j : TX = \bigoplus_{i=1}^r TX_i\to TX_i$, where $TX$ (resp. $TX_i$) denotes the holomorphic tangent bundle of $X$ (resp. $X_i$). This is the case when each $X_i$ is either a Calabi--Yau manifold (i.e. $c_1(X_i)=0$) or has ample canonical line bundle  $K_{X_i}$ and $L_i= K_{X_i}$, or is a Fano manifold  with vanishing Futaki invariant and $L_i= K^{-1}_{X_i}$.

\smallskip
Second,  it is now known that the extremal K\"ahler metrics in a K\"ahler class $\Omega$ are all isometric under the action of the {\it reduced automorphism group}\footnote{$\widetilde{\Aut}_0 (X)$ is the unique connected {\it linear
algebraic subgroup} of ${\Aut}_0 (X)$ such that the quotient ${\rm
Aut} _0 (X)/\widetilde{{\Aut}}_0 (X)$ is the Albanese torus of $X$~\cite{fujiki-0}; its Lie algebra is the space of
(real) holomorphic vector fields whose zero-set is non-empty
~\cite{fujiki-0,kobayashi,Le-Sim,gauduchon-book}.} $\widetilde{\Aut}_0 (X)$~\cite{BM,CT,Do-one,M4}. Thus,   the main difficulty in proving the splitting property is to show that if  the polarized projective manifold $(X,L)= \prod_{i=1}^r(X_i,L_i)$ admits an extremal K\"ahler metric, then each factor $(X_i,L_i)$ does also admit extremal K\"ahler metric. It was suggested by S.-T.~Yau \cite{yau-1} that a complete obstruction to the
existence of extremal K\"ahler metrics in the K\"ahler class $\Omega = 2\pi c _1
(L)$ on a projective manifold $X$ polarized by an ample
holomorphic line bundle $L$ should be expressed in terms of {\it stability}
of the pair $(X, L)$.  The currently accepted notion of stability is the
$K$-({\it poly}){\it stability} introduced by G.~Tian~\cite{Tian2} and S.~K.~Donaldson~\cite{Do2}.  The {\it Yau--Tian--Donaldson conjecture} can
then be stated as follows. {\it A polarized projective manifold $(X, L)$
admits a CSC K\"ahler metric if and only if it is $K$-polystable.}  The implication `CSC $\Rightarrow$
{K-polystable}' in the conjecture is now well-established, thanks to work
by S.~K.~Donaldson \cite{Do4}, X.~X.~ Chen--G.~Tian~\cite{CT}, J.~Stoppa
\cite{stoppa}, and T.~Mabuchi \cite{mab-three,mab-three1} but the other direction is still open.  In order to
account for extremal K\"ahler metrics of non-constant scalar curvature,
G.~Szekelyhidi  introduced in \cite{Sz, gabor} the notion of {\it
relative} $K$-(poly)stability with respect to a maximal torus of the
connected component ${\rm Aut}_0(X,L)$ of the automorphism group ${\rm Aut}(X,L)$ of the pair $(X, L)$~\footnote{Recall that ${\rm Aut}(X,L)$ consist of the automorphisms of $X$ which come from automorphisms of $L$. It is well-known (see e.g. \cite{kob-0,gauduchon-book}) that ${\rm Aut}_0(X,L)= \widetilde{{\rm Aut}}_0(X)$.} and the
similar implication `extremal $\Rightarrow$ {relative K-polystable}' was
obtained in~\cite{gabor-stoppa}. While it is not hard to see that in the product case (relative) $K$-(poly)stability of $(X,L)$ implies (relative) $K$-(poly)stability of each factor $(X_i, L_i)$, examples from~\cite{ACGT} suggest that the notion of relative $K$-(poly)stability must be further strengthened in order to establish the other direction in the Yau--Tian--Donaldson correspondence. 

\smallskip Our third observation is that if we start with a product K\"ahler metric in the class $2\pi c_1(L)$, invariant under a maximal connected compact subgroup $K$ of $\Aut_0(X)= \prod_{i=1}^r\Aut_0(X_i)$, then the $K$-relative Calabi flow (a gradient flow for the $K$-relative Mabuchi energy) preserves the Riemannian product structure. On the other hand, it is expected that this flow should converge to an extremal K\"ahler metric when it exists (see e.g.~\cite{D0}). Although this conjecture is very far from being solved, a partial evidence for it is given in~\cite{chen-hu,huang-zheng, tosatti-weinkove}. Note also that this approach has the advantage to apply to the more general case of a product of compact K\"ahler manifolds endowed with a product K\"ahler class.

\smallskip Thus motivated, we prove the splitting property under two additional hypotheses. 

\begin{thm}\label{main} Let $X_i$ be compact projective manifolds polarized by ample holomorphic line bundles $L_i$ and $X= \prod_{i=1}^r X_i$ their product endowed with the polarisation $L= \bigotimes_{i=1}^r L_i$, where $L_i$ is seen as a holomorphic line bundle over $X$ via the natural pull-back. Then, any extremal K\"ahler metric $g$ in the K\"ahler class $\Omega = 2\pi c_1(L)$ on $X$ is the Riemannian product of extremal K\"ahler metrics $g_i$ in the K\"ahler classes $\Omega_i=2\pi c_1(L_i)$ on the factors $X_i$, provided that at least one of the following hypotheses is satisfied. 
\begin{enumerate} 
\item[\rm (i)] The integral Futaki invariants of $(X,L)$ introdced in \cite{futaki-chow} all vanish.
\item[\rm (ii)] For at most one factor $(X_i,L_i)$, the group  ${\Aut}_0(X_i,L_i)$ has a center of positive dimension. 
\end{enumerate} 
\end{thm}

The hypothesis in (i) automatically holds if ${\rm Aut}_0(X,L)=\{ {\rm Id} \}$. However, it is known that the hypothesis in (i) is a restrictive condition in the case when ${\rm Aut}_0(X,L)$ is non-trivial (see e.g. \cite{OSY}). Also by the results in \cite{futaki-chow,M2}, in the case when $2\pi c_1(L)$ admits an extremal K\"ahler metric, $(X,L)$ is asymptotically Chow stable if and only if the integral Futaki invariants of $(X,L)$ introdced in \cite{futaki-chow} all vanish. More generally, the existence of an extremal K\"ahler metric in $2\pi c_1(L)$ is expected to imply that $(X,L)$ is asymptotically Chow stable with respect to a maximal torus $T \subset {\rm Aut}_0(X,L)$:  we give a precise formulation in Conjecture 1 below and discuss it in the light of  the work of T.~Mabuchi~\cite{M1,M2,M3,mab-three, mab-three1}.  We then show how the conjectured correspondence would solve (via Lemma 2 and Theorem 7) the splitting of the extremal K\"ahler metrics in the general polarized case.

\smallskip
We now outline the  proof of Theorem~\ref{main}. It uses an idea going back to G.~Tian~\cite{tian-0} (se also \cite{yau0}) who proved that any K\"ahler metric $\omega$ in $2\pi c_1(L)$ can be approximated with induced Fubini--Study metrics from the projective embeddings of the polarized variety $(X, L)$. More precisely, let $h$ be a hermitian metric on $L$ whose curvature is $\omega$. The induced hermitian metric on each tensor power $L^k$  is still denoted by $h$, and using $h$ and $\omega$,  consider the $L_2$ hermitian inner product on each vector space $H^0(X,L^k)$. Fixing an orthonormal basis for each $H^0(X,L^k)$, define a sequence of embeddings $\Phi_k : X \hookrightarrow  \C P^{N_k}$ and induced K\"ahler metrics $\frac{1}{k} \Phi_k^*(\omega_{\rm FS})$ in $\Omega=2\pi c_1(L)$. Tian showed that  $\frac{1}{k} \Phi_k^*(\omega_{\rm FS})$ converge to $\omega$  in $\cC^2$ as $k \to \infty$ while the $\cC^{\infty}$ convergence follows from subsequent work by W.~Ruan~\cite{ruan}. For each $k$, let  $\{ s_0, \cdots, s_{N_k} \}$ of $H^0(X,L^k)$ with respect to the $L_2$ hermitian inner product defined by $h_k=h^{\otimes k}$ and $\omega$, we denote the corresponding Bergman kernel $\rho_k$ as 
$$\rho_k = \sum_{i=0}^{N_k} h_k(s_i,s_i).
$$
The expansion of Bergman kernel was established by D.~Catlin~\cite{catlin} and S.~Zeldich~\cite{zeldich}. The coefficients of the expansion were calculated by Z. Lu \cite{Lu}. An important ramification of this basic idea, relevant to the problem of existence of CSC metric in $2\pi c_1(L)$,  was given by S.~K.~Donaldson~\cite{Do-one} who proved that when ${\Aut}_0(X)=\{ 1 \}$, a CSC K\"ahler metric $\omega$ in $2\pi c_1(L)$ can be approximated in $\cC^{\infty}$  by using special projective embeddings called {\it balanced}, a  notion  previously  introduced and studied by H.~Luo~\cite{L} and S.~Zhang~\cite{Z} (see also \cite{BLY}): a hermitian metric $h_k$ on $L^k$ is called balanced if the corresponding Bergman kernel $\rho_k$ is a constant function on $X$, or equivalently,  if the curvature $\omega_k$ of $h_k$ satisfies $\omega_k= \frac{1}{k} \Phi_k^*(\omega_{\rm FS})$. Thus, S.~K.~Donaldson's theorem states that if ${\Aut}_0(X,L)=\{ 1 \}$ and $\omega$  is a CSC K\"ahler metric in $2\pi c_1(L)$, then for $k \gg 1$, there exists a balanced hermitian metric $h_k$ on $L^k$ with curvature $\omega_k$ and, moreover,  $\omega_k$ converges to $\omega$ in $\cC^{\infty}$ as $k \to \infty.$ T.~Mabuchi~\cite{M1,M2,M3} extended Donaldson's result to the case when ${\Aut}_0(X,L)$  is non-trivial and $\omega$ is an extremal K\"ahler metric: in this case, $\omega$ can be approximated in $\cC^{\infty}$ by the normalized curvatures  $\omega_k$ of hermitian metrics $h_k$ on $L^k$ which are {\it balanced relative to a torus} in ${\rm Aut} _0 (X,L)$: this theory is reviewed in  Section~\ref{s:relative balanced}. For simplicity,  we shall momentarily refer to such $h_k$'s as {\it relative} balanced metrics on $L^k$. In the case when $(X,L)= \prod_{i=1}^r (X_i,L_i)$, Grauert's direct image theorem for coherent sheaves implies that $H^0(X,L^k) = \bigotimes_{i=1}^r H^0(X_i,L_i^k)$. It is then easily seen that if each $(X_i, L_i^k)$ admits a relative balanced hermitian metric, then the tensor product metric on $(X,L^k)$ is relative balanced and has curvature compatible with the product structure. Conversely, we show in Section~\ref{s:proof} (see Theorem~\ref{reduced}) that if $L^k$ admits some relative balanced metric then each $L_i^k$ does. We achieve this by studying in Section~\ref{s:functional} the Kempf--Ness function ${\mathbb D}$ introduced by H.~Luo~\cite{L} and S.~K.~ Donaldson~\cite{D1} (it is the function denoted $D$ in \cite{L} and $\tilde Z$ in \cite{D1} and is  essentially the $\log$ of the Chow norm introduced in \cite{Z}). This observation, together with Mabuchi's approximation result alluded to above,  reduces our problem to showing the uniqueness of relative balanced metric on $L$ modulo the action of ${\rm Aut}_0(X,L)$. This is not automatic in the setting of \cite{M1,M2,M3} but holds under the assumptions (i) or (ii) of Theorem~\ref{main}. We thus propose in Section~\ref{s:relative balanced} a stronger notion of relative balanced metrics (which also appears in the recent work \cite{M5}) and point out that for such (strongly) relative balanced metrics the uniqueness modulo ${\rm Aut}_0(X,L)$ automatically holds (Lemma~\ref{gabor}).

\section{Hermitian metrics balanced relative to a torus and relative Chow stability}\label{s:relative balanced}
In this section we briefly review some material taken from the works of S.~K.~Donaldson~\cite{Do-one,D1}, H.~Luo~\cite{L}, T.~Mabuchi~\cite{M1,M2,M3,M4,M5} and S.~Zhang~\cite{Z} that we shall need in the sequel.

Let $X$ be a  compact complex projective manifold of complex dimension $n$,  polarized by a  very ample line bundle $L$, and $N+1$ be the dimension of $V=H^0(X,L)$. Let $\kappa : X \hookrightarrow  P(V^*)$ denotes the Kodaira embedding with $L= \kappa^*(\cO(1))$. For any  basis  ${\bf s}=\{s_0, \cdots,  s_N\}$ of $V$  we denote 
$$
\Phi_{\bf s} : X \hookrightarrow  \mathbb{C} {P}^N 
$$
the composition of $\kappa$ with the identification ${\bf s} : P(V^*) \cong {\mathbb C} P^{N}$.

The reduced automorphisms group $\widetilde{\Aut}_0(X)$ is the closed connected subgroup of ${\rm Aut}_0(X)$ whose Lie algebra $\mathfrak{h}_0$ is the ideal of holomorphic vector fields with zeros on $X$, see \cite{fujiki-0,kobayashi,Le-Sim,gauduchon-book}. It is well-known (see e.g.~\cite{gauduchon-book,kob-0}) that $\widetilde{\Aut}_0(X)$ coincides with the connected component ${\Aut}_0(X,L)$ of the group of automorphisms of the pair $(X,L)$ and we obtain a group representation $\rho : {\Aut}_0(X,L) \to {\rm PGL}(V)$.   One can think of  ${\Aut}_0(X,L)$ as the connected group generated by restrictions to $\kappa (X)$ of elements of ${\rm PGL}(V)$ which preserve $\kappa(X) \subset P(V^*)$;  replacing $L$ by the tensor power $L^{N+1}$, we can further lift the action of $\widetilde{\Aut}_0(X)={\rm Aut}_0(X,L)$ on $X$ to an action on the bundle $L$ (see e.g. \cite{kob-0}),  and find a group representation
\begin{equation}\label{representation}
\rho : \mathrm{Aut}_0(X,L) \rightarrow {\rm SL}(V).
\end{equation}
In conclusion, by replacing $L$ with a sufficiently big tensor power if necessarily, we can assume that the reduced automorphisms group $\widetilde{\Aut}_0(X) = {\Aut}_0(X,L)$ of $X$ lifts to act on $L$,  and identify the action of $\widetilde{\Aut}_0(X) = {\Aut}_0(X,L)$ on $X$ with the induced action on $\kappa(X)$ of the connected subgroup ${\rm SL}_0(V, X)$ of elements ${\rm SL} (V)$ which preserve $\kappa(X) \subset P(V^*)$; furthermore, we shall also  assume $N>n$.

From now on, we shall fix a real torus $T \subset \widetilde{\Aut}_0(X,J)$ and consider hermitian metrics $h$ on $L$  which are $T$-invariant and whose curvature $\omega$ defines a $T$-invariant K\"ahler form in $2\pi c_1(L)$. Note also that, by the Calabi theorem~\cite{cal}, if the K\"ahler class $\Omega= 2\pi c_1(L)$ admits an extremal K\"ahler metric, it will also admits one which is $T$-invariant. Thus,  following \cite{M1}, we are now in position to introduce the notion of a ($T$-invariant) hermitian metric $h$ on $L$ which is balanced relative to $T$.  Denote by $T^{c}$ the complexified action of $T$ and consider the lifted  linear  $T^{c}$-action on $V$ via $\rho$.   Then, for every character $\chi \in \mathrm{Hom}(T^{c}, \mathbb{C}^*)$, we set
$$
V(\chi) := \{ s \in V; \rho(t)\cdot s = \chi(t)\ s \ \mbox{ for all } t \in T^{c} \},
$$
and obtain the splitting with respect to the mutually distinct characters $\chi_1, \ldots, \chi_{\nu} \in \mathrm{Hom}(T^{c}, \mathbb{C}^*)$  
\begin{equation}\label{e:split}
V = \bigoplus_{k=1}^{\nu} V(\chi_k),
\end{equation}
with $\prod_{k=1}^{\nu} \chi_k^{n_k}=1$ where $n_k= {\rm dim}_{\C}(V(\chi_k))$.

\begin{defn} Let $m(\cdot, \cdot)$ be a hermitian inner product on $V$.
We say that $\{s_0, s_1, \ldots, s_N \}$ is an $admissible\ normal\ basis$ of $(V,m)$ if it is compatible with the decomposition \eqref{e:split} and provides a normal basis of $m$ on each factor $V(\chi_k)$, i.e. if there exist positive real constants $b_k$ ($k=1, \ldots, \nu)$,  with $\sum_{k=1}^\nu n_k b_k = N+1$, and a sub-basis $\{s_{k,i},  \ k=1, \ldots, \nu, \ i=1, \ldots, n_k \}$ for $V(\chi_k)$, such that
\begin{enumerate}
\item[$\bullet$] $m(s_{k,i}, s_{l,j})=0$ if $l \neq k$ or $i\neq j$;
\item[$\bullet$] $m(s_{k,i}, s_{k,i}) = b_k.$
\end{enumerate}
The  vector $b := (b_1, \ldots, b_\nu)$ is called $index$ of the admissible normal basis $\{s_0, \ldots, s_N\}$ for $(V,m)$; in the case when the index is $b=(1, \ldots, 1)$ we shall call the basis {\it admissible orthonormal} basis of $(V,m)$.
\end{defn}
Note that a hermitian inner product $m( \cdot, \cdot)$ admits an admissible normal basis if and only if $V(\chi_k) \perp^m V(\chi_l)$ for $k\neq l$,  which in turn is equivalent to $m$ being $T$-invariant. For any $T$-invariant hermitian metric $h$ on $L$ whose curvature is a K\"ahler form  in $2\pi c_1(L)$,   $V(\chi_k) \perp V(\chi_l)$ for  $k \neq l$ with respect to the induced $L^2$ hermitian inner product $m=\langle \cdot, \cdot \rangle_h$ on $V$, defined by  
$$
\langle s_1, s_2 \rangle_h = \int_X h(s_1, s_2) \omega^n,
$$ 
for any two holomorphic sections $s_1, s_2 \in H^0(X, L)$.  We then define the smooth function $$E_{h, b} := \sum_{i=0}^N h(s_i,s_i),$$ which is clearly independent of the choice of an admissible normal basis of index $b$ on $(V, m)$.

\begin{defn}\label{d:balanced}
A $T$-invariant hermitian metric $h$ on $L$ whose curvature $\omega$ is K\"ahler metric on $X$ is called {\it balanced relative to $T$ of index $b$} if the function $E_{h, b}$ is constant for any admissible normal basis of index $b$. The curvature $\omega$ of $h$ is called a balanced K\"ahler metric of index $b$ relative to $T$.
\end{defn}

The definition above has the following useful interpretation in terms of the K\"ahler geometry of $X$. Consider the space ${\mathcal B}^T(V)$ of bases of $V=H^0(M,L)$, which are {\it compatible} with the splitting  \eqref{e:split}, i.e. which are admissible normal bases for some $T$-invariant hermitian inner product $m$. If ${\bf s}= \{ s_0, \cdots, s_N\}$ is an element of ${\mathcal B}^T(V)$ and $h$ is any $T$-invariant hermitian metric on $L$, we put 
\begin{equation}\label{hs}
h_{\bf s} =  \frac{h}{\sum_{i=0}^N h(s_i, s_i)},
\end{equation}
which is manifestly independent of the auxiliary hermitian metric $h$ on $L$. 

Any basis ${\bf s}=\{s_0, \cdots, s_N\}$  in ${\mathcal B}^T(V)$ determines a $T$-invariant hermitian inner product $m_{\bf s}$ on $V$ (and $V^*$) such that $\bf s$ (resp. the dual basis ${\bf s}^*$) is admissible and orthonormal. The identification ${\bf s}^* : P(V^*) \cong \C P^{N}$ determines a Fubini--Study metric $\omega_{\rm FS, {\bf s}}$ on $P(V^*)$,  representing $2\pi \cO(1)$;  we denote by $\omega_{X,{\bf s}}= \kappa^* (\omega_{\rm FS, {\bf s}})$ the induced K\"ahler form on $X$ via the Kodaira embedding $\kappa$. Note that $\omega_{X,{\bf s}}$ is the curvature of the hermitian metric $h_{\bf s}$ on $L$ defined by \eqref{hs} and if $\omega$ is the curvature of $h$, it is easily seen that 
\begin{equation}\label{potential}
\omega_{X,{\bf s}} = \omega + \frac{1}{2} dd^c \log (\sum_{i=0}^N h(s_i,s_i)).
\end{equation}
One therefore obatins
\begin{lemma}\label{characterization}
A $T$-invariant hermitian metric $h$ on $L$ is balanced relative to $T$ of index $b$  if and only if with respect to any admissible orthonormal basis ${\bf s}$ of the hermitian inner product $m_{h,b}$ on $V$,  defined by rescaling $\langle \cdot, \cdot \rangle_h$ on each space $V(\chi_k)$ by a factor $1/b_k^2$,  $h_{\bf s}= \lambda h$  for some positive constant $\lambda$. 
\end{lemma}

\smallskip
In order to give further motivation for the above notions, we now briefly recall the (finite dimensional) momentum map interpretation given by S.~K.~Donaldson~\cite{Do-one,D1}, and subsequently studied in \cite{PS,W}.

On the space ${\mathcal B}^T(V)$ the following groups act naturally:
\begin{enumerate}
\item[$\bullet$] ${\mathbb C}^{\times}$ by scalar multiplications;
\item[$\bullet$] $\rho(Z_{{\rm Aut}_0(X,L)}(T))$ where $Z_{{\rm Aut}_0(X,L)}(T)$ is the connected component of identity of the centralizer of $T$ in ${\rm Aut}_0(X,L)$;
\item[$\bullet$] $G_T={\rm S}(\prod_{k=1}^{\nu} {\rm U}(n_k))$, which is also a connected component of the centralizer  of $\rho(T)$  in ${\rm SU}(N+1)$.
\end{enumerate}
As the actions of ${\mathbb C}^{\times}$ and $\rho(Z_{{\rm Aut}_0(X,L)}(T))$ commute with the action of $G_T$, we can consider the quotient space ${\mathcal Z}^T(V) = {\mathcal B}^T(V) / \big({\mathbb C}^{\times} \times \rho(Z_{{\rm Aut}_0(X,L)}(T))\big)$ on which $G_T$ acts with stabilizer (of every point) $G_T \cap \rho(Z_{{\rm Aut}_0(X,L)}(T))$; in our setting $\rho(T) \subset  G_T \cap \rho(Z_{{\rm Aut}_0(X,L)}(T))$.  Following \cite{Do-one,PS},  there is a K\"ahler structure on ${\mathcal Z}^T(V)$, whose definition uses the fact that any point ${\bf s}= \{s_{k,i}, \ k=1,\ldots \nu,  \ i=1, \ldots, n_k\}$ of ${\mathcal B}^T(V)$ defines an embedding of $\Phi_{\bf s} : X \hookrightarrow {\mathbb C}P^N$. With respect to this K\"ahler structure $G_T$ acts isometrically with momentum map given (up to a non-zero multiplicative constant) by
$$\mu_{G_{T}}({\bf s}) = i \  \Big( \bigoplus_{k=1}^{\nu}  ( \langle s_{k,i}, s_{k,j}  \rangle_{h_{\bf s}}) \Big)_0, $$
where the $( \cdot )_0$ denotes the traceless part of the matrix (the Lie algebra ${\rm su}(N+1)$ being identified with its dual using the positive definite Killing form), and with complexification  $G_T^c= {\rm S}(\prod_{k=1}^{\nu}{\rm GL}(n_k, \C))$. It follows that ${\bf s}$ is a zero of the momentum map $\mu_{G_T}$ if and only if $h_s$ is a balanced metric of index $b=(1, \ldots, 1)$ relative to $T$; such a metric is also balanced with respect to the trivial torus $T=\{ {\rm Id} \}$ (of index $1$). This is the classic notion of balanced embedding studied in \cite{L,Z}. It follows from these works that the existence of a balanced basis ${\bf s}$ is equivalent to the Chow polystability of the variety $(X,L)$,  which we briefly recall: Let $d$ be the degree of the image $\kappa (X) \subset P(V^*)$ under the Kodaira embedding. Any element $h=(h_0, \cdots, h_{n})$ of $P(V) \times \cdots \times   P(V)$ ($(n+1)$-times) is seen as $(n+1)$ hyper-planes in $P(V^*)$, and
$$\{ h \in P(V) \times \cdots \times   P(V) : h_0 \cap \cdots \cap h_n \cap \kappa(X) \neq 0\}$$
becomes a divisor in $P(V) \times \cdots \times P(V)$ defined by a polynomial $\hat X \in W=\big({\rm Sym}^d(V^*)\big)^{\otimes (n+1)}$, called a {\it Chow line} and determined up to a non-zero scale; the corresponding element $[\hat X] \in P(W)$ is the {\it Chow point} associated to $(X,L)$.
\begin{defn}\label{chow} The polarized variety $(X,L)$ is called {\it Chow polystable} if the orbit of $\hat X$ in $W=\big({\rm Sym}^d(V^*)\big)^{\otimes (n+1)}$ under the natural action of ${\rm SL}(V)$ is closed. $(X,L)$ is called {\it asymptotically Chow polystable} if $(X,L^k)$ is Chow polystable for any $k\gg1$.
\end{defn}
The result of H.~Luo~\cite{L} and S.~Zhang~\cite{Z} (see also \cite[Theorem~A]{M1}) then states
\begin{thm}\label{luo-zhang} A compact polarized projective complex manifold $(X,L)$ is Chow polystable if and only if $L$ admits a balanced hermitian metric $h$ {\rm (}of index $1$ relative to $T=\{ {\rm Id} \}${\rm )}.
\end{thm}

The relevance of balanced metrics to our work comes from the following central result in the theory, proved by  S.~K.~Donaldson~\cite{Do-one} in the case when ${\Aut}_0(X,L)$  is trivial,  and  extended by T.~Mabuchi~\cite{M2} to the general case.
\begin{thm}\label{do-mabuchi} Let $(X,L)$ be an asymptotically Chow polystable compact polarized projective manifold and $\omega \in 2\pi c_1(L)$ a K\"ahler metric metric of constant scalar curvature~\footnote{A. Futaki~\cite{futaki-chow} showed that the extremal K\"ahler metrics in the K\"ahler class of an asymptotically Chow stable polarization must be CSC.}. Then, there exist sequences of integers $m_k \to \infty$  and hermitian metrics $h_k$ on $L^{m_k}$,  with curvatures $\omega_k$, which are balanced {\rm (}relative to $T=\{ 1 \}$ of index $b=1${\rm )},  and such that $\frac{1}{m_k}\omega_k$ converge  in $\cC^{\infty}$ to $\omega$ as $k \to \infty$. \end{thm}
Note that when ${\rm Aut}_0(X,L)$ is trivial, S.~K.~Donaldson also shows in \cite{Do-one} that the existence of a CSC K\"ahler metric in $2\pi c_1(L)$ implies that $(M,L)$ is asymptotically Chow polystable while it is known that the latter condition is restrictive in the case when ${\rm Aut}_0(X,L)$ is non-trivial (see e.g. \cite{OSY}). 

\smallskip
One therefore needs to further relax the condition on balanced metrics in order to find similar approximations of extremal K\"ahler metrics on  asymptotically Chow unstable varieties, and this is where the choice of  indices $b$ will come to play. Using the momentum map picture described above,  a natural approach developed in \cite{gabor, Sz}  would be,  instead of zeroes of $\mu_{G_T}$,  to study the critical points of the squared norm $||\mu_{G_T} ||^2$ (with respect to the positive definite  Killing inner product of $\mathfrak{su}(N+1)$). It follows from the moment map picture that a basis ${\bf s} \in {\mathcal B}^T(V)$ is a critical point of $||\mu_{G_T}||^2$ if and only if $\mu_{G_T}({\bf s})$ is a matrix which belongs to the Lie algebra of the stabilizer of the projection of ${\bf s}$ to ${\mathcal Z}^{T}$ for the action of $G_T$. In order to simplify the discussion, and with the application in mind, let us assume that $T$ is a maximal torus in ${\rm Aut}_0(X,L)$.  This implies that  $\rho(Z_{{\rm Aut}_0(X,L)}(T)) \cap G = \rho(T)$ i.e.  the stabilizer of any point of ${\mathcal Z}^T(V)$ is $\rho(T)$.  Therefore, a basis ${\bf s}$ is a critical point for $||\mu_{G_T}||^2$ if and only if $\mu_{G_T}({\bf s})$ is a diagonal matrix $$ i  \ {\rm diag}(a_1, \ldots, a_1, a_2, \ldots, a_2, \cdots, a_{\nu}, \ldots a_{\nu})$$ which belongs to the Lie algebra of $\rho(T)$. In other words, ${\bf s}$ defines a critical point of $||\mu_{G_T}||^2$ on ${\mathcal Z}^{T}$ if and only if the induced hermitian metric $h_{\bf s}$ on $L$ is balanced relative to $T$ of index $b=(b_1, \cdots, b_{\nu})$ with 
\begin{equation}\label{constraint}
b_k= \frac{1 +  \log |\chi_k(t)|}{1 + \sum_{\ell=1}^{\nu} \frac{n_{\ell} \log |\chi_{\ell}(t)|}{N+1}},  \ k=1, \ldots,  \nu
\end{equation} 
for some $t \in T^c$. 
The corresponding interpretation in terms of Chow stability has been worked out by T.~Mabuchi~\cite{M5} and is expressed by the closeness in $W$ of the Chow line $\hat X$ under the natural action of the group
$$G_{T^{\perp}} ^c (V)= \{ {\rm diag} (A_1, \cdots, A_{\nu}) \in \prod_{k=1}^{\nu} {\rm GL}(V({\chi_k})) :  \prod_{k=1}^{\nu} {\rm det} (A_k)^{1+\log |\chi_k(t)|}=1 \ \forall t \in T^c \}.$$
\begin{defn}\label{stability}  We call $(X,L)$ {\it Chow polystable stable relative to $T$} if the Chow line $\hat X$ associated to $(X,L)$ has a closed orbit with respect to $G_{T^{\perp}}^c(V)$; $(X,L)$ is called {\it asymptotically Chow stable relative to $T$} if  $(X, L^k)$ is  Chow polystable stable relative to $T$  for all $k\gg 1$. 
\end{defn}
\noindent We then have (cf. \cite[Theorem A]{M1} and \cite[Theorem C]{M5}) 
\begin{thm}\label{chow-stability}
A polarized  compact projective complex manifold $(X,L)$ is Chow polystable relative to $T$ if and only if $L$ admits a hermitian metric balanced relative to $T$ for some index $b$ satisfying \eqref{constraint}. 
\end{thm}

A generalization of the Kempf--Ness theorem (see \cite[Theorem 3.5]{gabor} or \cite[Theorem 1.3.4]{Sz}) in the momentum map set up provides us with the following useful result.
\begin{lemma}\label{gabor} Let $(X,L)$ be a compact polarized projective complex manifold and $T$ a maximal torus in ${\rm Aut}_0(X,L)$. Then $L$ admits a hermitian metric balanced relative to $T$ of some index $b$ satisfying \eqref{constraint} if and only if the orbit of some {\rm (}and hence each{\rm )} point of ${\mathcal B}^T(V)$ under the action of the group $$G^c_{T^\perp}=\{ {\rm diag} (A_1, \cdots, A_{\nu}) \in \prod_{k=1}^{\nu} {\rm GL}(n_k, \C)) :  \prod_{k=1}^{\nu} {\rm det} (A_k)^{1+\log |\chi_k(t)|}=1 \ \forall t \in T^c \}$$ contains a basis ${\bf s}$ such that $h_{\bf s}$ is balanced with respect to $T$ of index satisfying \eqref{constraint}. Furthermore, any two balanced hermitian metrics  relative to $T$ with indices satisfying \eqref{constraint} are homothetic under the action of ${\rm Aut}_0(X,L)$. \end{lemma}
It is not difficult to give a direct prove of Lemma~\ref{gabor},   once one knows the relevant identities to use. The uniqueness part follows from the fact that the $T^c$ action generates balanced metrics relative to $T$  (of some index $b$ satisfying \eqref{constraint}) and the corresponding admissible bases of  index $b$ (see Lemma~\ref{characterization}) exhaust the $G^c_{T^\perp}$ orbits of ${\mathcal Z}^T(V)$; one can then apply Proposition~\ref{convex} in Section~\ref{s:functional}.  In particular, the index $b$ in Lemma~\ref{gabor} is uniquely determined.

\smallskip
In view of the discussion above, the following provides a natural scope of a generalization of Theorem~\ref{do-mabuchi}.
\begin{conj}\label{relative-stability} Let $(X,L)$ be a  compact polarized projective manifold and $\omega \in 2\pi c_1(L)$ an extremal K\"ahler metric which, without loss, is  invariant under a maximal torus $T\subset {\rm Aut}_0(X,L)$.  Then $(M,L)$ is asymptotically Chow polystable relative to $T$,  and there exists a sequence of integers $m_k \to \infty$ and $T$-invariant hermitian metrics $h_k$ on $L^{m_k}$ with curvatures $\omega_k$, which are balanced relative to $T$ of indices $b_k$ satisfying \eqref{constraint},  such that the corresponding relative balanced K\"ahler metrics $\frac{1}{m_k}\omega_k$ on $X$ converge in $\cC^{\infty}$ to $\omega$ as $k \to \infty$.
\end{conj}

\smallskip
In a series of work \cite{M1,M2,M3}, T.~Mabuchi has established a weaker version of Conjecture~\ref{relative-stability}. The main idea is to consider instead of the group $G_T$, the smaller group $G = \prod_{k=1}^{\nu} {\rm SU}(n_k)$ which acts on ${\mathcal Z}^T(V)$ with momentum map
$$\mu_{G}({\bf s})= i \  \bigoplus_{k=1}^{\nu}  \Big( \langle s_{k,i}, s_{k,j}  \rangle_{h_{\bf s}}\Big)_0, $$
so that the zeroes of $\mu_G$ correspond to bases ${\bf s}$ in ${\mathcal B}^T(V)$ for which the hermitian metrics $h_{\bf s}$ on $L$ which are balanced relative to $T$ of some index $b$ (not necessarily satisfying \eqref{constraint}). The corresponding notion of Chow stability is then
\begin{defn}\label{weak-stability}  $(X,L)$ is {\it weakly Chow polystable stable relative to $T$} if the Chow norm $\hat X$  associated to $(X,L)$ has a closed orbit under the action of $G^c(V)= \prod_{k=1}^{\nu} {\rm SL}(V(\chi_k))$. We call $(X,L)$ {\it asymptotically weakly Chow polystable stable relative to $T$} if $(M,L^k)$ is weakly Chow polystable stable relative to $T$ for all $k\gg 1$. 
\end{defn}

The following result is extracted from \cite{M1,M2,M3}.
\begin{thm}\label{mabuchi} Let $(X,L)$ be a compact polarized projective manifold and $\omega \in 2\pi c_1(L)$ an extremal K\"ahler metric which, without loss, is invariant under a maximal compact connected subgroup $K\subset  {\Aut}_0(X,L)$. Let $T\subset K$  be any  torus in the connected component of the identity of the centre of  $K$. Then, $(M,L)$ is asymptotically weakly Chow polystable relative to $T$ and there exists a sequence of integers $m_k \to \infty$ and $T$-invariant hermitian metrics $h_k$ on $L^{m_k}$ with curvatures $\omega_k$, which are balanced relative to $T$,  such that $\frac{1}{m_k}\omega_k$ on $X$ converge in $\cC^{\infty}$ to $\omega$ as $k \to \infty$. 
\end{thm}
The above statement is implicitly established in \cite{M3} in the course of proof that the existence of an extremal K\"ahler metric in $2\pi c_1(L)$ implies  that $(X,L)$ is asymptotically weakly Chow polystable relative to $T$; the choice of $T$ is specified by \cite[Theorem~I]{M3}. More precisely,  \cite[Theorem~B]{M1} shows that any $T$-invariant extremal K\"ahler metric in $2\pi c_1(L)$ can be approximated by a sequence of {\it almost critical} metrics; then, combining S.~K.~Donaldson's idea in \cite{Do-one} and D.~Phong--J.~Sturm's estimates in \cite{PS}, a perturbation technique is elaborated in \cite{M2} and applied in \cite{M3} in order to perturb the almost critical metrics to balanced metrics relative to $T$,  in a way that their curvatures converge to $\omega$.

The limitation of Theorem~\ref{mabuchi} to complete the proof of the splitting properly (Theorem~\ref{main} in the introduction) in full generality is in the lack of analogue of Lemma~\ref{gabor}, which guarantees that any two balanced metrics relative to $T$ on $L$ are homothetic. We show that this is true under the hypothesis (i) and (ii) of Theorem~\ref{main}.

\section{The Kempf--Ness function ${\mathbb D}$}\label{s:functional}
In this section we are going to apply the well-known `Kempf--Ness' principle related to the problems of studying zeroes of momentum maps. For simplicity, we discuss the existence of hermitian metrics on $L$ which are balanced relative to a fixed torus $T\subset {\rm Aut}_0(X,L)$ of some index, but the discussion and all of the results can be easily adapted to the case of indices satisfying \eqref{constraint} simply by changing the group $G$ to $G_{T^{\perp}}$. We have seen in Section~\ref{s:relative balanced} that the problem of finding a basis ${\bf s}\in {\mathcal B}^T(V)$ for which $h_{\bf s}$ is balanced with respect to $T$ is equivalent to finding zeroes of the momentum map $\mu_G$ in a given orbit $G^c \cdot [{\bf s}_0 ]\subset {\mathcal Z}^T(V)$. As $\mu_G$ is $G$-equivariant, this becomes a problem on the symmetric space $G^c/G$. On that space we are going to consider a function $F_{{\bf s}_0} : G^c/G \to {\mathbb R}$, called {\it Kempf--Ness} function, whose behaviour determines whether or not there exists a zero of $\mu_G$ on $G^c \cdot {\bf s}_0$. This function is geodesically convex and its  derivative  is essentially $\mu_G$; hence  $\mu_G$  admits a zero on $G^c \cdot {\bf s}_0$ if and only if $F_{{\bf s}_0}$ attains a minimum on $G^c/G$.

On the space ${\mathcal H}$ of all hermitian inner products $m$ on $V$ such that  $V(\chi_k) \perp^{m} V(\chi_l), \ l \neq k$ (equivalently, which admit admissible normal bases of some index)  the group $G^c(V)$ acts with stabilizer $G(V,m)=G^c(V)  \cap {\rm U}(V,m)$; thus, for each $m_0 \in {\mathcal H}$, by introducing an admissible orthonormal basis ${\bf s}_0$,  we can identify the corresponding orbit ${\mathcal M}_{m_0}= G^c(V) \cdot m_0$ with the symmetric space $G^c/G$ (which is known to be reducible of non-positive sectional curvature). The underlying riemannian metric is explicitly given by (see e.g.~\cite{helgason})
\begin{equation}\label{metric}
(M_1, M_2)_m = {\rm Tr}(M_1 \cdot m^{-1} \cdot M_2 \cdot m^{-1}),
\end{equation}
where the hermitian inner product $m$ is identified with a positive-definite hermitian endomorphism of $V$  via $m_0$, and $M_1,M_2 \in T_m({\mathcal M}_{m_0})$ with hermitian skew-symmetric endomorphisms of $(V, m_0)$.  Another well-known fact (see e.g. \cite{helgason})  is that  geodesics correspond to $1$-parameter subgroups of $G^c(V)$, so the geodesic $m(t)$ joining two points $m_1, m_2 \in {\mathcal M}_{m_0}$  is generated by the family of admissible normal bases ${\bf s}(t)= \{e^{t\gamma_0}s_0, \cdots, e^{t\gamma_N}s_N\}$,  where ${\bf s}=\{s_0, \cdots, s_N\}$ is an admissible orthonormal basis for $m_1$ which diagonalizes $m_2$, and $m_2(s_i,s_i)= e^{-2\gamma_i}$ (with $\sum_{i=1}^{n_k} \gamma_{i,k}=0$) and $m(t)$ is the unique hermitian inner product for which $s(t)$ is an admissible orthonormal  basis.

Denote by $\mathcal{K}_\omega$ be the set of all K\"ahler metrics  in the K\"ahler class $[\omega]$, i.e.
$$
\mathcal{K}_\omega = \{ \omega_\varphi \ |  \ \omega_\varphi = \omega + dd^c \varphi > 0, \varphi \in C^\infty (X) \}
$$
We can define a map $\mathcal{FS} : \mathcal{H} \mapsto \mathcal{K}_\omega$ as follows:  For any $m \in {\mathcal H}$  let ${\bf s}=\{s_0, \cdots, s_N\}$ be an admissible orthonormal basis of $V$  and $\omega_{FS,{\bf s}}$ the Fubini--Study it defines on $P(V^*)$. Consider the pull-back $\omega_{X,{\bf s}}= \kappa^* (\omega_{\rm FS, {\bf s}})$ under the Kodaira embedding (satisfying \eqref{potential}),  which is the curvature of the hermitian metric $h_{\bf s}$ on $L$, given by \eqref{hs}. Put
\begin{eqnarray}
\label{w}
\mathcal{FS}(m) : = \omega_{X, {\bf s}},  \  \  h_{m}:= h_{\bf s}, 
\end{eqnarray}
noting that for a fixed $m$ the right hand sides of \eqref{hs} and \eqref{potential} are  independent of the choice of orthonormal basis ${\bf s}$.

Many authors have considered (see e.g. \cite{do,gauduchon-book}) the functional  ${\mathbb I} : \mathcal{K}_{\omega} \to \R$, defined up to an additive constant  by requiring that its derivative $\delta {\mathbb I}$ is given by
$$
(\delta {\mathbb I}) (\dot{\varphi}) = \int_X \dot{\varphi} \ \omega_\varphi^n,
$$
where $\dot{\varphi} \in T_{\omega_{\varphi}} (K_{\omega}) = C^{\infty}(X)$.  Following H.~Luo~\cite{L} and S.~K.~Donaldson~\cite{D1},  we introduce ${\mathbb D} : \mathcal{H} \to {\mathbb R}$ by 
\begin{equation} \label{d:d}
{\mathbb D} (m) : = - {\mathbb I} (\mathcal{FS} (m)).
\end{equation}
The restriction of ${\mathbb D}$ to the $G^c(V)$ orbit ${\mathcal M}_{m_0}$  defines a function on $G^c(V)/G(V,m_0)$ which, by introducing an admissible orthonormal basis ${\bf s}_0$ of $m_0$,  will be the Kempf--Ness function $F_{{\bf s}_0} : G^c/G \to {\mathbb R}$ referred to earlier.

\smallskip
The following results in this section are essentially proved in  \cite{D1} and \cite{L}. The way we treat the reduced automorphism group is inspired by X.~X.~Chen's work \cite{C1}. 

We start by  characterizing  the critical points of ${\mathbb D}$.

\begin{prop} \label{cha}
A hermitian inner product $m$ is a critical point of ${\mathbb D} : G^c(V)\cdot m_0 \mapsto \R$  if and only if the induced hermitian metric $h_m$ defined by \eqref{hs} and \eqref{w} is balanced {\rm (}of some index $b${\rm )} with respect to $T$, i.e. if and only if  any admissible orthonormal basis ${\bf s}$ of $m$ is a zero of the momentum map $\mu_G$. Likewise,  $m$ is a critical point of ${\mathbb D} : G^c_{T^{\perp}}(V)\cdot m_0 \mapsto \R$ if and only if $h_m$ is balanced of index satisfying \eqref{constraint}.
\end{prop}

\begin{proof}
For ${\mathcal M}_{m_0} = G^c(V) \cdot m_0$, we will prove that $m$ is a critical point of ${\mathbb D}: {\mathcal M}_{m_0} \mapsto \R$ if and only if there exist real numbers $(b_1, \ldots, b_{\nu})$ such that for some (and hence any) admissible orthonormal basis  ${\bf s}= \{ s_{k,i}, 1 \leq k \leq \nu, 1 \leq i \leq n_k \}$ of $m$
\begin{equation}\label{conditions}
\begin{split}
\int_X h_{m} (s_{k,i}, s_{k,i}) \ {\mathcal FS}(m)^n & =   b_k, \  i=1, \ldots, n_k, \ k=1, \ldots, \nu \\
\int_X h_{m} (s_{k,i}, s_{l,j}) \ {\mathcal FS}(m)^n & = 0,  \  {\rm if} \ k\neq l \ {\rm or} \  i \neq j.
\end{split}
\end{equation}
Using the hermitian metric $h_m$ in the definition \eqref{hs}, we see that the conditions \eqref{conditions} are equivalent to $h_m$ being balanced relative to $T$ of index $b=(b_1, \ldots, b_{\nu})$. The case of a $G^c_{T�{\perp}}(V)$ orbit will follow with obvious modifications of the arguments below.

\smallskip
($\Rightarrow$) Let ${\bf s}= \{s_{k,i}, 1 \leq k \leq \nu, 1 \leq i \leq n_k\}$ be an admissible orthonormal basis of $m$. For any choice of $\gamma_{l,j}$  with 
$$
\sum_{i=1}^{n_k} \gamma_{k,i} = 0.
$$
the basis ${\bf s}_t=\{ e^{\gamma_{l,j}t} s_{l,j} \}$ defines a hermitian inner product $m(t)$ on $V$ (such that ${\bf s}_t$ is an  admissible orthonormal bases for $m(t)$) and, as we have noticed,  $m(t)$ is a geodesic. Put ${\mathbb D}(t)={\mathbb D}(m(t))$. Using \eqref{potential}, \eqref{w}  and \eqref{d:d},  we obtain for the derivative ${\mathbb D}'(t)$
\begin{equation}\label{D'}
\begin{split}
{\mathbb D}'(t) &= \int_X \frac{\sum_{j=0}^N 2\gamma_j e^{2t\gamma_j}|s_j|_h^2}{\sum_{j=0}^N e^{2t\gamma_j}|s_j|_h^2} \Big({\mathcal FS}(m(t))\Big)^n  \\
&= \int_X Q(t) \Big( {\mathcal FS}(m(t))\Big)^n,
\end{split}
\end{equation}
with 
\begin{equation}\label{Q}
Q(t) = \frac{\sum_{j=0}^N 2\gamma_j e^{2t\gamma_j}|s_j|_h^2}{\sum_{j=0}^N e^{2t\gamma_j}|s_j|_h^2}.
\end{equation} 
Then,  the fact that $m$ is a critical point of ${\mathbb D}$ implies 
\begin{equation}\label{calcul}
0 ={\mathbb D}'(0) = 2 \sum_{k=1}^{\nu}\sum_{i=1}^{n_k} \gamma_{k,i}  \int_X h_{m} (s_{k,i}, s_{k,i}) \ {\mathcal FS}(m)^n,
\end{equation}
where we have used the fact that \eqref{Q} is independent of the choice of a hermitian metric $h$ on $L$. For the latter equality to hold for any choice of real numbers $\gamma_{k,i}$ as above, $\int_X h_{m} (s_{k,i}, s_{k,i}) \ {\mathcal FS}(m)^n$ must be 
independent of $i$; since ${\mathcal FS}(m)$  and $h_{m}$  are $T$ invariant, $\int_X h_{m}(s_{k,i}, s_{l,j}) \ {\mathcal FS}(m)^n = 0$ for $k \neq l$. Elementary linear algebra shows that if $\int_X h_{m} (s_{k,i}, s_{k,i}) \ {\mathcal FS}(m)^n$ is independent of $i$ for {\it any} choice of an admissible orthonormal basis ${\bf s}$ for $m$,  then one must have $\int_X h_{m} (s_{k,i}, s_{k,j}) \ {\mathcal FS}(m)^n = 0$ for $ i \neq j$.

($\Leftarrow$) The conditions \eqref{conditions}  mean that some admissible orthonormal basis ${\bf s}$ of $m$ is an admissible normal basis (of index ${b}=(b_1, \ldots, b_{\nu})$) for the induced $L_2$ hermitian inner product $\langle \cdot , \cdot \rangle_{h_{m}}$; this is clearly independent of the choice of a particular admissible orthonormal basis of $m$. It is therefore enough to pick one admissible orthonormal basis  ${\bf s}$, and show that if \eqref{conditions} is satisfied, then $m$ must be a critical point of ${\mathbb D}$, or equivalently, ${\mathbb D}'(0)=0$ along any geodesic $m(t)$ issued at $m$. The computation \eqref{calcul} shows this. \end{proof}

\begin{rem}\label{chow-norm} If we consider ${\mathbb D}$ as a function on $G^c(V)/G(V,m_0)$ (or $G^c_{T^{\perp}}(V)/G_{T^{\perp}}(V, m_0)$), the computation \eqref{D'} (compared to a similar result  in \cite{Z}) shows that ${\mathbb D}$ coincides, up to a positive scale and an additive constant, with the function $\log || \cdot ||_{{\rm CH},m_0}$ defined on $G^c(V)/G(V,m_0)$  (resp. $G^c_{T^{\perp}}(V)/G_{T^{\perp}}(V, m_0)$), where $|| \cdot ||_{{\rm CH},m_0}$ is the $U(V,m_0)$-invariant  {\it Chow norm} on the space $W= {\rm Sym}^d(V^*)^{\otimes (n+1)},$ introduced in \cite{Z}. This, together with Proposition~\ref{cha},  explains Theorems~\ref{luo-zhang} and \ref{chow-stability}, once one proves (as in \cite{Z} and \cite{M1}) that $\log || \cdot ||_{{\rm CH},m_0}$ has a critical point on the $G^c(V)$ (resp. $G^c_{T^{\perp}}(V)$) orbit of $\hat X$ if and only the orbit is closed (i.e. $(X,L)$ is (weakly) relative Chow stable). Note that the latter condition is independent of the choice of $m_0$, showing that the existence of critical points of ${\mathbb D}$ is independent of the choice of a $G^c(V)$ orbit in ${\mathcal H}$.
\end{rem}

The next results hold for ${\mathcal M}_{m_0}$ being either the  $G^c(V)$ or the $G^c_{T^{\perp}}(V)$ orbit of $m_0$ in ${\mathcal H}$.

\begin{prop}\label{convex} 
${\mathbb D}$ is convex along geodesics in $\mathcal{M}_{m_0}$. Furthermore, for any two critical points $m_1, m_2$ {\rm (}if they exist{\rm )}, the geodesic $m(t)$ joining $m_1$ and $m_2$  defines a family of balanced hermitian metrics $h_{m(t)}$ relative to $T$ on $L$, which are isometric under the action of ${\rm Aut}_0(X,L)$ and, therefore, have the same index $b$.
\end{prop}
\begin{proof} By \eqref{D'}, the second derivative of ${\mathbb D}(t)$ is
\begin{eqnarray*}
{\mathbb D}''(t) &=& \int_X \Big(\frac{\partial Q}{\partial t} + n |dQ|^2_{{\mathcal FS}(m(t))} \Big)  \mathcal{FS}(m(t))^n.
\end{eqnarray*}
To show the convexity of ${\mathbb D}(t)$ along geodesics, we adopt an argument of T.~Mabuchi~\cite{M1} by constructing the map
\begin{eqnarray*}
\eta : [0,1] \times [0, 2\pi) \times X &\rightarrow& \mathbb{C}P^N,\\
(t, \theta, x) &\mapsto& [e^{(t+\sqrt{-1}\theta)\gamma_0} s_0(x), \ldots, e^{(t+\sqrt{-1}\theta)\gamma_N} s_N(x)].
\end{eqnarray*}
Letting  $z=t+\sqrt{-1}\theta$, we have
\begin{eqnarray*}
0 & \leq & \int_{\eta([0,1] \times [0, 2\pi) \times X)} \omega_{{\rm FS},{\bf s}}^{n+1}\\
&=&\int_{[0,1] \times [0, 2\pi) \times X} (\eta^*\omega_{{\rm FS}, {\bf s}})^{n+1}\\
&=&\int_{[0,1] \times [0, 2\pi) \times X} ( \frac{\sqrt{-1}}{2\pi} \partial_X \bar{\partial}_X \log (\sum_{j=0}^N e^{2t\gamma_j} |s_j|^2 ) + \frac{\sqrt{-1}}{4\pi} \partial_X Q \wedge d \bar{z}\\
& & + \frac{\sqrt{-1}}{4\pi} d z \wedge \bar{\partial}_X Q + \frac{\sqrt{-1}}{8\pi} \frac{\partial Q}{\partial t} dz \wedge d\bar{z})^{n+1} \\
&=& 2\pi \int_0^1 \int_X \frac{n+1}{4\pi} \frac{\partial Q}{\partial t} + \frac{(n+1)n}{4\pi} |d Q|^2_{{\mathcal FS}(m(t))} (\Phi^*_t \omega_{{\rm FS}, {\bf s}(t)})^n dt \\
&=& \frac{n+1}{2} \int_0^1 {\mathbb D}^{''}(t) dt
\end{eqnarray*}
Hence ${\mathbb D}(t)$ is convex. It follows that any critical point of ${\mathbb D}$ is a global minimizer. Suppose now $m_1, m_2$ are two minimizers of ${\mathbb D}$. Joining $m_1$ and $m_2$ with a geodesic $m(t)$ as in the proof of \eqref{D'}, we have ${\mathbb D}(0) = {\mathbb D}(1)$ and ${\mathbb D}'(0) = 0 = {\mathbb D}'(1)$,  so that ${\mathbb D}''(t) = 0, 0 \leq t \leq 1$. It follows from the calculation above that
$$
\int_{\eta([0,1] \times [0, 2\pi) \times X)} \omega_{{\rm FS}, {\bf s}}^{n+1} = 0.
$$
We conclude that for any point $p \in \eta([0,1] \times [0, 2\pi) \times X) \subset \mathbb{C}P^N$, the complex dimension of any neighbourhood of $p$ in $\eta([0,1] \times [0, 2\pi) \times X)$ is $n$. Hence the image  $\Phi_{{\bf s}(t)}(X)$ is fixed, showing that  the one-parameter group ${\rm diag}(e^{t\gamma_0}, \ldots, e^{t\gamma_N})$ in ${\rm SL}(N+1, \C)$ induces a one-parameter group in $\widetilde{\Aut}_0(X) = {\Aut}(X,L)$. \end{proof}

For any $m \in {\mathcal M}_{m_0}$, we introduce the group
$$\Aut_m (X, {\mathcal M}_{m_0}, {\mathbb D}) = \{ g \in \widetilde{\Aut}_0(X)  \ | \  \rho(g) ({\mathcal M}_{m_0}) = {\mathcal M}_{m_0},  \ {\mathbb D} (\rho (g) \cdot m) = {\mathbb D}(m) \},$$
where $\rho$ is the representation \eqref{representation}. Clearly, ${\Aut}_m(X, {\mathcal M}_{m_0}, {\mathbb D})$ is a closed subgroup of ${\Aut}(X,L)= \widetilde{\Aut}_0(X)$  while $\rho({\Aut}_m(X, {\mathcal M}_{m_0}, {\mathbb D}))$ is a closed subgroup of ${\rm SL}(V)$.

\begin{lemma} For any two points $m_1, m_2 \in {\mathcal M}_{m_0}$, 
$ \mathrm{Aut}_{{m_1}}(X, {\mathcal M}_{m_0}, {\mathbb D}) = \mathrm{Aut}_{m_2}(X, {\mathcal M}_{m_0}, {\mathbb D}).$
\end{lemma}

\begin{proof}
Let $m(t), 0 \leq t \leq 1$ be the geodesic connecting $m_1$ and $m_{2}$ and ${\bf s}$ an admissible orthonormal basis with respect to $m_1$. For any $g \in \mathrm{Aut}_{{m_1}}(X, {\mathcal M}_{m_0}, {\mathbb D})$, $\rho(g) \cdot m(t)$ is the geodesic connecting $\rho(g) \cdot m_1$ and $\rho(g) \cdot m_{2}$. Using the integral formula \eqref{D'}, and noting that \eqref{Q} is independent of the choice of a hermitian metric $h$,   we have $\frac{d}{dt}{\mathbb D}(m(t)) = \frac{d}{dt} {\mathbb D}(\rho(g) \cdot m(t))$, and therefore ${\mathbb D}(m_{2}) = {\mathbb D}(\rho(g) \cdot m_{2})$. Hence $g \in \mathrm{Aut}_{m_{2}}(X, {\mathcal M}_{m_0}, {\mathbb D})$. Similarly, $\mathrm{Aut}_{m_{2}} (X, {\mathcal M}_{m_0}, {\mathbb D})\subset \mathrm{Aut}_{m_1}(X, {\mathcal M}_{m_0}, {\mathbb D})$. \end{proof}
In view of the above lemma, we adopt 

\begin{defn}  ${\Aut}_X({\mathcal M}_{m_0}, {\mathbb D})$ is the closed subgroup of $\widetilde{\Aut}_0(X)$  of elements which preserve ${\mathcal M}_{m_0}$ and ${\mathbb D}$.
\end{defn}

\begin{rem} By definition, $\rho({\rm Aut}_X({\mathcal M}_{m_0}, {\mathbb D})) \subset \rho(Z_{{\Aut}_0(X,L)}(T))\cap G_{T^{\perp}}^c$. If $T$ is a maximal torus in ${\Aut}_0(X,L)$ and $X$ admits an extremal K\"ahler metric, a result by E.~Calabi~\cite{cal} implies  that $Z_{{\Aut}_0(X,L)}(T)= T^c$. We conclude that in this case ${\rm Aut}_X({\mathcal M}_{m_0}, {\mathbb D})$ is trivial.  

Formula \eqref{D'} shows that any element of $\rho(Z_{{\Aut}_0(X,L)}(T))\cap G_T^c$ (resp. $\rho(Z_{{\Aut}_0(X,L)}(T))\cap G_{T^{\perp}}^c$) sends a critical point of ${\mathbb D}$ to another critical point. It then follows from Proposition~\ref{convex} that when ${\mathbb D}$ atteins its minimum on ${\mathcal M}_{m_0}$ (i.e. $(X,L)$ is (weakly) relative Chow stable, see Theorem~\ref{chow-stability} and Remark~\ref{chow-norm}), ${\rm Aut}_X({\mathcal M}_{m_0}, {\mathbb D})$ is the sub-group of $Z_{{\Aut}_0(X,L)}(T)$ of elements whose lifts by $\rho$ belong to $G^c_T$ (resp. $G^c_{T^{\perp}}$). 
\end{rem}
\begin{lemma}\label{l:2} Suppose that ${\mathbb D}$ has a minimum on ${\mathcal M}_{m_0}$. Then, the set of all minimizers represents an orbit for the induced action $\rho(\Aut_X({\mathcal M}_{m_0}, {\mathbb D}))$ and
for any   $m \in \mathcal{M}_{m_0}$ there exists a minimizer $m_{\mathrm{min}}$ of ${\mathbb D}$ such that 
$$
d(m,m_{\mathrm{min}}) = \min_{ g \in \mathrm{Aut}_X({\mathcal M}_{m_0}, {\mathbb D})} d(m, \rho(g) \cdot m_{\mathrm{min}}),
$$
where $d$ is the distance function defined on $\mathcal{M}_{m_0}$ with respect to the metric \eqref{metric}. Furthermore, if $m(t), 0 \leq t \leq 1$ is the geodesic connecting $m$ and $m_{\mathrm{min}}$, then 
$$
d(m(t),m_{\mathrm{min}}) = \min_{g \in \mathrm{Aut}_X({\mathcal M}_{m_0}, {\mathbb D})} d(m(t), \rho(g) \cdot m_{\mathrm{min}}).
$$
\end{lemma}

\begin{proof}
The first part follows from Proposition~\ref{convex}. For the second claim, suppose $g_k \in \Aut_X({\mathcal M}_{m_0}, {\mathbb D})$ is a sequence such that  $$\lim_{k\to \infty} d(m, \rho(g_k)\cdot m_{\rm min}) = \inf_{g \in \mathrm{Aut}_X({\mathcal M}_{m_0}, {\mathbb D})} d(m, \rho(g) \cdot m_{\rm min}).$$ Let us denote $m_k= \rho(g_n) \cdot m_{\rm min}$ and choose an admissible normal basis ${\bf s}^k$ of $m$ which diagonalizes $m_k$. As $G(V)$ (resp. $G_{T^{\perp}}(V)$) is compact, we can assume that  ${\bf s}^k$ converges to an admissible normal basis  ${\bf s}$ of $m$. On the other hand, as in the proof of Proposition~\ref{convex},  we can express the geodesic between $m$ and $m_k$ by using a one parameter subgroup of $G^c(V)$ (resp. $G^c_{T^{\perp}}(V)$) generated by  ${\rm diag}(e^{\gamma^k_0}, \cdots , e^{\gamma^k_N})$ and compute
$d^2(m,m_k)= \sum_{i=0}^N |\gamma_i^k|^2,$
so that, taking a subsequence, ${\rm diag}(e^{\gamma^k_0}, \cdots , e^{\gamma^k_N})$ converges to  a diagonal matrix ${\rm diag}(e^{\gamma_0}, \cdots , e^{\gamma_N})$; it defines an element $m_{\infty} \in {\mathcal M}_{m_0}$  such that $m_{\infty}(s_i,s_j)=0$ for $i\neq j$ and  $m_{\infty}(s_i,s_i)= e^{-2\gamma_i}m(s_i,s_i)$. The last conclusion holds easily by using the triangle inequality.\end{proof}
The next result establishes the properness of ${\mathbb D}$,  provided it has critical points on ${\mathcal M}_{m_0}$. A similar result has been originally established by H.~Luo \cite{L} in the case when ${\widetilde \Aut}_0(X)$ is trivial. 
\begin{prop}\label{proper} Suppose $m_{\min}$ is a minimizer of ${\mathbb D}$ on $\mathcal{M}_{m_0}$. For every $C > 0$, there exists $C_1$ such that for any $m \in \mathcal{M}_{m_1}$ with the property
$$
{\mathbb D}(m) < {\mathbb D}(m_{\min}) + C,
$$
there exists a $g \in \mathrm{Aut}_X({\mathcal M}_{m_0}, {\mathbb D})$ such that
$$
d(m, \rho(g) \cdot m_{\min}) < C_1.
$$
\end{prop}

\begin{proof}
Suppose for  contradiction that there is a constant $C > 0$ and a sequence $m_i \in \mathcal{M}_{m_0}$ such that 
\begin{equation}\label{contradiction}
{\mathbb D}(m_i) < {\mathbb D}(m_{\min}) + C
\end{equation}
and
\begin{equation}\label{contradiction1}
d(m_i, \rho(g) \cdot m_{\min}) > i
\end{equation}
for any $g \in \mathrm{Aut}_X({\mathcal M}_{m_0}, {\mathbb D})$.  By Lemma~\ref{l:2},  for any $i$ there exists $g_i \in \Aut_X({\mathcal M}_{m_0}, {\mathbb D})$ such that
$$
d(m_i, \rho(g_i) \cdot m_{\min}) = \min_{g \in \mathrm{Aut}_X({\mathcal M}_{m_0}, {\mathbb D})} d(m_i, \rho(g) \cdot m_{\min})> i.
$$
Let $m_i(t), 0 \leq t \leq d_i=d(m_i, m_{\min})$ be the normal geodesic connecting $m_{\min}$ and $m_i$. Then, using \eqref{contradiction}, \eqref{contradiction1} and Proposition~\ref{convex}, we get
\begin{eqnarray*}
C & > & {\mathbb D}(m_i) - {\mathbb D}(m_{\min}) \\
& = & \int_0^{d_i} {\mathbb D}'(m_i(t)) dt\\
& \geq & \int_0^i {\mathbb D}'(m_i(t)) dt\\
& \geq & i \int_0^1 {\mathbb D}'(m_i(t)) dt\\
&=& i({\mathbb D}(m_i(1)) - {\mathbb D}(m_{\min})).
\end{eqnarray*}
Letting ${\tilde m}_i = m_i(1)$, we have 
$$
{\mathbb D}({\tilde m}_i) < {\mathbb D}({\tilde m}_{\min}) + \frac{C}{i}
$$
while, by Lemma~\ref{l:2}, 
$$
1 = d({\tilde m}_i, m_{\min}) = \min_{g \in \mathrm{Aut}_X({\mathcal M}_{m_0}, {\mathbb D})} d({\tilde m}_i, \rho(g) \cdot m_{\min}).
$$
Taking a subsequence of ${\tilde m}_i$ converging to a minimizer $m_\infty$ of ${\mathbb D}$, we obtain a contradiction (see Lemma~\ref{l:2}). \end{proof}

\section{Proof of Theorem~\ref{main}}~\label{s:proof}
It is enough to consider the case when the polarized projective manifold $(X,L)$ is the product of two factors $(X_1,L_1)$ and $(X_2,L_2)$. Denote the dimensions of $X, X_1, X_2$ by $n, n_1, n_2$ respectively. Letting $p_i : X \to X_i$ be the canonical projections, we have $L = \pi_1^*(L_1) \otimes \pi_2^*(L_2)$. 

The holomorphic splitting of the tangent bundle $TX = TX_1 \oplus TX_2$ induces a product structure $\widetilde{\Aut}_0(X) = \widetilde{\Aut}_0(X_1)\times \widetilde {\Aut}_0(X_2)$,  so we can fix a maximal torus $T \subset \widetilde {\Aut}_0(X)$ of the form $T= T_1 \times T_2$,  where $T_i \subset \widetilde {\Aut}_0(X_i)$ are maximal tori.  Taking a common tensor power of the $L_i$'s  if necessarily, we will suppose that $(X,L)$ and $(X_i,L_i)$  all satisfy the assumptions made in Section~\ref{s:relative balanced}.  Grauert's direct image theorem for coherent sheaves implies that $V = V_1 \otimes V_2$ where $V=H^0(X,L)$ and $V_i= H^0(X_i, L_i)$. Notice that if $V_i$ splits under $T_i$ as
$$
V_i = \bigoplus_{k=1}^{\nu_i} V_i(\chi^i_k),
$$
then
$$
V = \bigoplus_{j,k} V_1 (\chi^1_j) \otimes V_2(\chi^2_k)
$$
gives the decomposition \eqref{e:split} for $V$ with $\chi_j^1\otimes \chi_i^2 = \chi_k$.

Let $m^i_{0}$ be $T_i$-invariant hermitian inner products on $V_i$. Simplifying the notation in Section~\ref{s:functional},  we let
$\mathcal{M}_i$  be the $G_i^c$ (resp. $({G_i})^c_{T_i^{\perp}})$) orbit  of $m^i_{0}$. The tensor product (of hermitian inner products and bases)  defines a  natural map $\mathcal{M}_1 \times \mathcal{M}_2 \to {\mathcal M}$ where ${\mathcal M}$ is the $G^c$ (resp. $G^c_{T^{\perp}}$) orbit of  $m_0=m^1_{0}\otimes m^2_{0}$. We define the subspace ${\mathcal M}_{\rm prod}$ of decomposable elements  of $\mathcal{M}$  
$$
\mathcal{M}_{\rm prod}  = \{m =m^1 \otimes m^2, \  | \ m^1 \in \mathcal{M}_1, m^2 \in \mathcal{M}_2 \}.
$$

\begin{lemma} \label{subspace} ${\mathcal M}_{\rm prod}$ is a closed totally geodesic submanifold of $\mathcal{M}$ which is stable under the action of $\rho(\widetilde{\Aut}_0(X)) \cap G^c$. Furthermore, for each $m=m^1\otimes m^2 \in  {\mathcal M}_{\rm prod}$  the induced metric ${\mathcal FS}(m)= {\mathcal FS}(m^1) + {\mathcal FS}(m^2)$ on $X=X_1\times X_2$ is a product metric.
\end{lemma}

\begin{proof} As $\widetilde{\Aut}_0(X)= \widetilde{\Aut}_0(X_1) \times \widetilde{\Aut}_0(X_2)$ and we have assumed (by taking a tensor power of $L_i$)  that each $\widetilde{\Aut}_0(X_i)$ acts on $L_i$, it follows  that  $\rho(\widetilde{\Aut}_0(X)) \cap G^c$ preserves ${\mathcal M}_{\rm prod}$.

From the description of the geodesics  of ${\mathcal M}$ (resp. ${\mathcal M}_i$) in terms of a 1-parameter subgroups of $G^c$ (resp. $G_i^c$) used in the proof of Proposition~\ref{convex},  it follows that if $m^i(t)$ is a geodesic of ${\mathcal M}_i$ ($i=1,2$), then $m(t)=m^1(t)\otimes m^2(t)$ is a geodesic of ${\mathcal M}$ which belongs to ${\mathcal M}$. 

Thus, in order to established the first part of Lemma~\ref{subspace}, we only need to show that ${\mathcal M}_{\rm prod}$ is a closed 
subset of ${\mathcal M}$. Consider a  sequence $m_k = m^1_k \otimes m^2_k \in {\mathcal M}_{\rm prod}$ with  $m^i_k \in \mathcal{M}_i$. The expression of  the geodesic joining $m^i_{0}$ and $m_k^i$ in terms of a 1-parameter subgroup ${\rm diag}(e^{t\gamma^i_0}, \cdots e^{t\gamma^i_{N_i}})$ of $G_i^c$  (see the previous section) allows us to compute the distance functions $d$ and $d_i$
\begin{equation*}
\begin{split}
d_i(m^i_{0}, m^i_k) ^2& = \sum_{j=0}^{N_i} (\gamma^i_j)^2,  \\
d(m_0, m_k)^2  & = \sum_{r=0}^{N_1} \sum_{j=0}^{N_2} (\gamma^1_r + \gamma^2_j)^2= N_2 d_1(m_{h_1}^1,m^1_k)^2 + N_1 d_2(m_{h_2}^2,m^2_k)^2,
\end{split}
\end{equation*}
where we have used that $\gamma^i_j$ satisfy $\sum_{j=0}^{N_i} \gamma_j^i=0$ for $i=1,2$.  This completes the first part of the Lemma.

The final claim is a direct consequence of \eqref{potential} and the fact that if we have chosen $h=h_1\otimes h_2$ where $h_i$ is a $T_i$-invariant hermitian metric on $L_i$, then the curvature is $\omega= \omega_1 + \omega_2$. \end{proof}

\begin{prop}\label{split}
For any critical point $m$ of ${\mathbb D}$ on ${\mathcal M}$, the induced K\"ahler metric  on $X=X_1 \times X_2$  is compatible with the the product structure. 
\end{prop}
\begin{proof} Any critical point of ${\mathcal M}$ must necessarily be a minimizer by Proposition~\ref{convex}.  Let $m_{\min} $ be such a minimizer. We pick a sequence $m_k \in {\mathcal M}_{\rm prod}$ such that
$$
\lim_{k \rightarrow \infty} \mathbb{D} (m_k) = \inf_{m \in {\mathcal M}_{\rm prod}} \mathbb{D} (m).
$$
Since the functional ${\mathbb D}$ defined on $\mathcal{M}$ is proper in the sense of Proposition~\ref{proper}, there exist  $g_i \in \mathrm{Aut}_X({\mathcal M}, {\mathbb D})$ such that 
$$
d(\rho(g_i)^{-1} \cdot m_{\min}, m_i) = d(m_{\min}, \rho(g_k) \cdot m_k) < C_1
$$
for all $i$. Putting ${\tilde m}_i = \rho(g_i) \cdot m_i$,  we know by Lemma~\ref{subspace} that $ {\tilde m}_i \in {\mathcal M}_{\rm prod}$. Taking a convergent subsequence of ${\tilde m}_k$  and using the closeness of ${\mathcal M}_{\rm prod}$ (see Lemma~\ref{subspace}), there exists $m \in {\mathcal M}_{\rm prod}$ such that
$$
{\mathbb D}(m) = \min_{\bar{m} \in {\mathcal M}_{\rm prod}} {\mathbb D}(\bar{m}).
$$

Let $m=m^1\otimes m^2$ be a minimizer of ${\mathbb D}$ on ${\mathcal M}_{\rm prod}$. We claim that  $m^i$ is a critical point of the corresponding functional ${\mathbb D}_i$ on $\mathcal{M}_i$. Without loss of generality, we only check this  for $m^1$. Suppose $m^1(t)$ is a geodesic starting from $m^1(0)=m^1$ in $\mathcal{M}_1$, expressed in terms of a 1-parameter subgroup  ${\rm diag}(e^{t\gamma_0^1}, \cdots, e^{t \gamma_{N_1}^1})$ of $G_1^c$ (resp. $(G_1)^c_{T_1^{\perp}}$): there exists an admissible orthonormal basis ${\bf s}^1 =\{ s^1_i, 0 \leq i \leq N_1\} $ of $m^1$  such that ${\bf s}^1(t)=\{e^{t\gamma_0^1}s^1_1, \ldots, e^{t\gamma_{N_1}^1}s_{N_1} \}$ is an admissible orthonormal basis for $m^1(t)$. Let ${\bf s}^2 =\{ s^2_j, 0 \leq j \leq N_2\} $ be an admissible orthonormal basis for $m^2$. Then $m(t) = m^1(t) \otimes m^2$ is a geodesic in $\mathcal{M}$ starting from $m$ and ${\bf s}^1(t) \otimes {\bf s}^2$ is an admissible orthonormal basis  for $m(t)$. Since $m(t) \in {\mathcal M}_{\rm prod}$ and $m$ is a minimizer of $\mathbb{D}$ on ${\mathcal M}_{\rm prod}$, we have
\begin{eqnarray*}
0 &=& \mathbb{D}'(0)\\
&=& \int_X \frac{\sum_{i=0}^{N_1} \sum_{j=0}^{N_2} 2 \gamma_i^1 |s^1_i|^2_{h_1} |s^2_j|^2_{h_2}}{\sum_{i=0}^{N_1} \sum_{j=0}^{N_2} |s^1_i|^2_{h_1} |s^2_j|^2_{h_2}}  \mathcal{FS}(m)^{n_1+n_2}\\
&=& \big( \int_{X_1} \frac{\sum_{i=0}^{N_1} 2 \gamma_i^1 |s^1_i|^2_{h_1} }{\sum_{i=0}^{N_1} |s^1_i|^2_{h_1}}  \mathcal{FS}(m^1)^{n_1}\Big) \times \Big( \int_{X_2} \frac{\sum_{j=0}^{N_2} |s^2_j|^2_{h_2}}{ \sum_{j=0}^{N_2} |s^2_j|^2_{h_2}}  \mathcal{FS}(m2)^{n_2}\Big)\\
&=& C \int_{X_1} \frac{\sum_{i=0}^{N_1} 2 \gamma_i^1 |s^1_i|^2_{h_1} }{\sum_{i=0}^{N_1} |s^1_i|^2_{h_1}}  \mathcal{FS}(m^1)^{n_1}=  C  \ \mathbb{D}_1'(0),
\end{eqnarray*}
where $C$ is a strictly positive constant. We  conclude that $m^1$ is a critical point of $\mathbb{D}_1$ on $\mathcal{M}_1$ by using Proposition~\ref{cha}. Conversely, Proposition~\ref{cha} also shows that $m$ is a critical point of ${\mathbb D}$ on $\mathcal{M}$.  Now, by Proposition~\ref{convex}, the induced K\"ahler metrics on $X$ by the critical points of ${\mathbb D}$ are isometric under the action of $\widetilde{\Aut}_0(X)= \widetilde{\Aut}_0(X_1) \times \widetilde{\Aut}_0(X_2)$ so, in particular, to the induced product K\"ahler metric by $m=m^1\otimes m2$ (see Lemma~\ref{subspace}), which completes the proof.\end{proof}

As the existence of critical points of ${\mathbb D}$ is independent of the choice of orbits (see Lemma~\ref{gabor} and Remark~\ref{chow-norm}), we obtain as an immediate corollary of Proposition~\ref{split}
\begin{thm}\label{reduced} Suppose $X$ admits a balanced K\"ahler metric relative to $T$ in $2\pi c_1(L)$. Then there exits a balanced K\"ahler metric relative to $T$ in $2\pi c_1(L)$ compatible with the product structure  $X=X_1\times X_2$.
\end{thm}

\noindent{\it Proof of Theorem~\ref{main}.} Combining Theorem~\ref{reduced} with Theorem~\ref{do-mabuchi} and Propositions 1 and 2 yields the proof of Theorem~\ref{main}(i). In order to prove Theorem~\ref{main}(ii), we use Theorem~\ref{mabuchi} with $T$ being the connected component of the centre of $\widetilde{{\rm Aut}}_0(X)$, so that, by the assumption,  for one of the factors, $(X_1,L_1)$ say, $T_1=\{ {\rm Id} \}$. It is not hard to see that in this case {\it each} $G^c$ orbit of admissible hermitian inner products on $V=V_1\otimes V_2$ contains products $m=m^1\otimes m^2$. (The latter is not true in general.) We can then apply Proposition~\ref{split}. $\Box$

\begin{rem} The above arguments and the uniqueness established in Lemma 2 would imply the splitting property should Conjecture \ref{relative-stability} be true.
\end{rem}

\end{document}